\documentclass[11pt,reqno]{amsart}
\usepackage{latexsym,amssymb,amsfonts,mathrsfs}
\usepackage{amsthm}
\usepackage{amsmath}
\usepackage[a4paper,bottom=2.5cm]{geometry}
\usepackage{a4wide}
\usepackage{enumerate}
\usepackage{comment}
\usepackage[latin1, utf8]{inputenc}
\usepackage{eufrak}
\usepackage[mathcal]{euscript}

\usepackage[usenames]{color}
\usepackage{graphicx}
\usepackage{lmodern,xfrac,xcolor}
\usepackage[noadjust]{cite}
\usepackage[unicode]{hyperref}

\usepackage{mathtools}
\usepackage[toc]{appendix}
\mathtoolsset{showonlyrefs}
\usepackage{scalerel}

\def\e{\mathrm{e}}
\def\ve{\varepsilon}
\def\la{\lambda}
\def\me{\mathsf{e}}
\def\mv{\mathsf{v}}
\def\mE{\mathsf{E}}
\def\mV{\mathsf{V}}
\def\mG{\mathsf{G}}

\def\ea{\EuFrak{a}}

\def\ve{\varepsilon}
\def\diag{\mathrm{diag}}

\def\comp{\mathbb{C}}
\def\real{\mathbb{R}}
\def\nat{\mathbb{N}}
\def\mcA{\mathcal{A}}

\def\mcB{\mathscr{B}}

\def\mcV{\mathcal{V}}

\def\mcS{\mathcal{S}}

\def\kF{\kappa_{\scaleto{\widetilde{F}}{4.5pt}}}
\def\kG{\kappa_{\scaleto{G}{3.3pt}}}

\def\dx{\, dx}
\def\dy{\, dy}
\def\dt{\, dt}
\def\ds{\, ds}

\newtheorem{theo}{Theorem}
\newtheorem{lemma}[theo]{Lemma}
\newtheorem{prop}[theo]{Proposition}
\newtheorem{cor}[theo]{Corollary}
\newtheorem{defi}[theo]{Definition}
\newtheorem{asum}[theo]{Assumption}
\newtheorem{assum}[theo]{Assumptions}
\newtheorem{rem}[theo]{Remark}

\newtheorem{notat}[theo]{Notation}

\numberwithin{equation}{section} \numberwithin{theo}{section}

\begin{document}

\title{Reaction-diffusion equations on metric graphs with edge noise}

\author{Eszter Sikolya}
\address{Department of Applied Analysis and Computational Mathematics\\
E\"otv\"os Lor\'and University\\
Budapest, Hungary}
\email{eszter.sikolya@ttk.elte.hu}
\thanks{The author was supported by the OTKA grant no.~135241.}

\date{\today}

\subjclass[2010]{Primary: 60H15, 35R02, Secondary: 35R60, 47D06}
\keywords{Metric graph; reaction-diffusion equation; Wiener-type noise}

\begin{abstract}
We investigate stochastic reaction-diffusion equations on finite metric graphs. On each edge in the graph a multiplicative cylindrical Gaussian noise driven reaction-diffusion equation is given.  The vertex conditions are the standard continuity and generalized, non-local Neumann--Kirchhoff-type law in each vertex. The reaction term on each edge is assumed to be an odd degree polynomial, not necessarily of the same degree on each edge, with possibly stochastic coefficients and negative leading term. The model is a generalization of the problem in \cite{KS21} where polynomials with much more restrictive assumptions are considered and no first order differential operator is involved. We utilize the semigroup approach from \cite{KS21JEE} to obtain existence and uniqueness of solutions with sample paths in the space of continuous functions on the graph. 
\end{abstract}

\maketitle

\pagestyle{plain}

\section{Introduction}

Throughout the paper $\mG$ denotes a finite metric graph. Our terminology follows \cite[Chap.~1]{BeKu}, we list here only the most important concepts. The graph $\mG$ consists of a finite set of vertices $\mV = \{\mv\}$ and a finite set $\mE = \{\me\}$ of edges connecting the vertices. We denote by $m=|\mE|$ the number of edges and by  $n=|\mV|$ the number of vertices. In general, a metric graph is assumed to have directed edges; that is edges having an origin and a terminal vertex. In our case, dealing with self-adjoint operators in the deterministic part of our model, we can just consider undirected edges. Each edge is assigned a positive length $\ell_{\me}\in (0,+\infty)$, and we denote by $x\in [0,\ell_{\me}]$ a coordinate of $\mG$. We assume that $\mG$ is simple; that is, there are no multiple edges connecting two vertices, and there are no loops at any of the vertices in $\mG$. 

The metric graph structure enables us to speak about functions $u$ on $\mG$, defined along the edges such that for any coordinate $x$, the function takes its value $u(x)$. If we emphasize that $x$ is taken from the edge $\me$, we write $u_{\me}(x)$. Hence, a function $u$ on $\mG$ can be regarded as a vector of functions that are defined on the edges, therefore we will also write
\[u=\left(u_{\me}\right)_{\me\in \mE},\]
and consider it as an element of a product function space.

To write down the vertex conditions in the form of equations, for a given function $u$ on $\mG$ and for each $\mv\in\mV$, we introduce the following notation. For any $\mv\in\mV$, we denote by $\mE_{\mv}$ the set of edges incident to the vertex $\mv$, and by $d_{\mv}=|\mE_{\mv}|$ the degree of $\mv$. Let $u_{\me}(\mv)$ denote the value of $u$ in $\mv$ along the edge $\me$, in the case $\me\in \mE_{\mv}$. Let $\mE_{\mv}=\{\me_1,\dots ,\me_{d_{\mv}}\}$, and define
\begin{equation}\label{eq:Fv}
U(\mv)=\left(u_{\me}(\mv)\right)_{\me\in \mE_{\mv}}=\begin{pmatrix}
	u_{\me_1}(\mv)\\
	\vdots\\
	u_{\me_{d_{\mv}}}(\mv)
\end{pmatrix}\in \real^{d_{\mv}},
\end{equation}
the vector of the function values in the vertex $\mv$. 

Let $I_{\mv}$ be the bi-diagonal matrix
\begin{equation}\label{eq:Iv}
I_{\mv}=\begin{pmatrix}
	1 & -1 & & \\
	& \ddots & \ddots & \\
	& & 1 & -1 
	\end{pmatrix}\in\real^{(d_{\mv}-1)\times d_{\mv}}.
\end{equation}
It is easy to see that if we set
\begin{equation}\label{eq:contv}
I_{\mv}U(\mv)=0_{\real^{d_{\mv}-1}},
\end{equation} 
this means that all the function values coincide in $\mv$. If this is the case, we simply write $u(\mv)$ for this common vertex value. 

\begin{defi}\label{defi:contG}
If the function $u$ is continuous on the edges, that is, $u_{\me}$ is continuous on $\me$ for each $\me\in\mE$ and \eqref{eq:contv} is satisfied for each vertex $\mv\in\mV$, then we call the function $u$ \emph{continuous on $\mG$}. 
\end{defi}

We define the operator with domain consisting of continuous functions on $\mG$ as
\begin{equation}\label{eq:opL}
\begin{split}
D(L) &\coloneqq \left\{u\in\prod_{\me\in\mE}C[0,\ell_{\me}]\colon I_{\mv}U(\mv)=0,\;\mv\in\mV\right\},\\
Lu &\coloneqq \left(u(\mv)\right)_{\mv\in\mV}\in\real^{n}\text{ for }u\in D(L).
\end{split}
\end{equation} 
That is, $L$ assigns to each function $u$ that is continuous on $\mG$ the ($n$ dimensional) vector of the vertex values of $u$.\medskip

For $T>0$ given we consider the stochastic system written formally as
\begin{equation}\label{eq:stochnet}
\left\{\begin{aligned}
\dot{u}_{\me}(t,x) & =  (c_{\me} u_{\me}')'(t,x)+d_{\me}(x)\cdot u_{\me}'(t,x)&&\\
&-p_{\me}(x) u_{\me}(t,x)+f_{\me}(t,x,u_{\me}(t,x))&&\\
&+ h_{\me}(t,x,u_{\me}(t,x))\frac{\partial w_{\me}}{\partial t}(t,x), &x\in (0,\ell_{\me}),\; t\in(0,T],\; \me\in\mE,\;\; & (a)\\
0 & = I_{\mv}U(t,\mv),\;   &t\in(0,T],\; \mv\in\mV,\;\;& (b)\\
0 & = M L u(t)+ Cu(t), & t\in(0,T],\;\;& (c)\\
u_{\me}(0,x) & =  u_{0,\me}(x), &x\in [0,\ell_{\me}],\; \me\in\mE.\;\;& (d)
\end{aligned}
\right.
\end{equation}
Here $\dot{u}_{\me}$ and $u_{\me}'$ denote the time and space derivatives, respectively, of $u_{\me}$.

The functions $c_{\me}$ are (variable) diffusion coefficients or conductances, and we assume that
\[0<c_{\me}\in C[0,\ell_{\me}],\quad  \me\in\mE.\]

We assume
\begin{equation}
d_{\me}\in \mathrm{Lip}[0,\ell_{\me}],\quad \me\in\mE.
\end{equation}

The functions $p_{\me}$ are nonnegative, bounded functions, hence
\begin{equation}\label{eq:pe}
0\leq p_{\me}\in L^{\infty}(0,\ell_{\me}),\quad \me\in\mE.
\end{equation}

The reaction terms $f_{\me}$  are assumed to be odd degree polynomials, with possible different degree on different edges, and with possibly stochastic coefficients and negative leading term, see \eqref{eq:fjdef}. The coefficients $h_{\me}$ are assumed to be locally Lipschitz continuous and satisfy appropriate growths conditions \eqref{eq:hjdefnov}, depending on the maximum and minimum degrees of the polynomials $f_{\me}$ on the edges. These become linear growth conditions when the degrees of the polynomials $f_{\me}$ on the different edges coincide. The $(w_{\me}(t))_{t\in [0,T]}$ are independent cylindrical Wiener-processes defined in the Hilbert spaces $L^2(0,\ell_{\me})$, $\me\in\mE$. 

For each $t\in [0,T]$, by $u(t)$ we mean
\[u(t)=u(t,\cdot)=\left(u_{\me}(t,\cdot)\right)_{\me\in\mE}\]
the function on $\mG$ defined as $u_{\me}(t,\cdot)$ on the edge $\me$. 

In $(\ref{eq:netcp}b)$, $0$ denotes the constant $0$ vector of dimension $d_{\mv}-1$ for $\mv\in\mV$, and $U(t,\mv)$ is defined by \eqref{eq:Fv} for the function $u(t)$. Hence, these equations express that $u(t)$ is continuous on $\mG$ for each $t\in (0,T]$, cf.~Definition \ref{defi:contG}.

In $(\ref{eq:stochnet}c)$, $0$ denotes the constant $0$ vector of dimension $n$, $M$ is an $n\times n$ matrix, satisfying certain conditions that we specify later, see Assumption \ref{as:M}, and $L$ is the operator from \eqref{eq:opL}. 

To define the so-called feedback operator $C$, 
for each $\mv\in\mV$ we set
\begin{equation}
C_{\mv} u\coloneqq\sum_{\me\in\mE_{\mv}}c_{\me}(\mv)\cdot u'_{\me}(\mv),
\end{equation}
where the derivatives are taken in the directions away from the vertex $\mv$ (i.e.~into the edge), see \cite[Sec.~1.4.]{BeKu}.
Let now
\begin{equation}\label{eq:opC}
\begin{split}
D(C) &\coloneqq \prod_{\me\in\mE}C^1[0,\ell_{\me}],\\
Cu &\coloneqq \left(C_{\mv}u \right)_{\mv\in\mV}\in \real^{n}\text{ for }u\in D(C).
\end{split}
\end{equation}
Hence, equation $(\ref{eq:stochnet}c)$ expresses a (generalized) Neumann--Kirchhoff-type condition in the vertices.

It is clear by definition, that $(\ref{eq:stochnet}b)$ consists of 
\begin{equation}\label{eq:conteqnumber}
\sum_{\mv\in\mV}(d_{\mv}-1)=2m-n
\end{equation} 
equations. At the same time, $(\ref{eq:stochnet}c)$ consists of $n$ equations. Hence, we have altogether $2m$ (boundary or vertex) conditions in the vertices.

In equation~$(\ref{eq:stochnet}d)$ we pose the initial conditions on the edges.\medskip

In the literature stochastic (reaction-)diffusion equations on networks are treated e.g.~in \cite{BMZ08}, \cite{BZ14}, \cite{CP17}, \cite{CP17a} and \cite{KS21JEE}. In the first four papers the semigroup approach is utilized in a Hilbert space setting. In the recent work \cite{KS21JEE} a much more general problem is treated with multiplicative Wiener type noise on the edges as well as in the vertices and with locally Lipschitz continuous diffusion coefficients satisfying appropriate growths conditions. This paper uses an entirely different tool-set based on the semigroup approach for stochastic evolution equations in Banach spaces from \cite{Ce03}, \cite{KvN12}, \cite{KvN19} and \cite{vNVW08}. This approach makes also possible to consider polynomial nonlinearities with different degrees on different edges. 

In the current paper we consider a generalization of the stochastic reaction-diffusion problem on a metric graph from \cite{KS21}. In \cite{KS21} the results are proved under rather restrictive conditions on the polynomial nonlinearities, namely, they have the same degree on the edges and their coefficients should be contained in the space $B$ of continuous functions on the graph $\mG$, see \eqref{eq:B}. Now we consider general nonlinearities of odd-degree polynomials, with possible different degree on different edges and with possibly stochastic coefficients, see \eqref{eq:fjdef}. We also add a first-order term in $(\ref{eq:stochnet}a)$, and consider different edge lengths.

Our equations differ from those in \cite{KS21JEE} since we have no dynamics, and correspondingly no noise in the vertices. Hence, in contrary to \cite{KS21JEE}, the state space of the problem will consist of functions on the edges and no boundary space is included. To work in this new setting we have to introduce the product space of continuous functions on the edges $E^c$, see \eqref{eq:Ec} and verify results concerning this space. A crucial point is to prove Proposition \ref{prop:fractionalspaceinclB} claiming the existence of continuous, dense embeddings of the fractional domain spaces of the generators $A_p$ (see \eqref{eq:Ap}) into the space $B$. In contrary to \cite[Lem.~4.2]{KS21JEE}, we do not have isometry here. Techniques and results from \cite{KS21JEE} make us possible to show in Theorem \ref{theo:SCPnsolcontB} that for any initial value from $E^c$ problem \eqref{eq:stochnet} admits a unique mild solution with trajectories in the space $B$ which is a more general result than obtained in \cite[Thm.~3.15]{KS21} for the Allen--Cahn type nonlinearities.

The paper is organized as follows. In Section \ref{sec:determnetwork} we collect known semigroup results for the linear deterministic version of \eqref{eq:stochnet} and generalize them for the case of different edge lengths. Section \ref{sec:srden} contains the core of the paper. In Subsection \ref{sec:srdenabstr} we recall  from \cite{KS21JEE} an abstract result regarding the stochastic abstract Cauchy problem \eqref{eq:SCP} on Banach spaces. In Subsection \ref{sec:srdenprep} we introduce the spaces $E^c$ and $B$ in Definitions \ref{defi:Ec} and \ref{defi:B}, and prove the embedding results Proposition \ref{prop:fractionalspaceinclB} and Corollary \ref{cor:fractionalspaceincl}. Proposition \ref{prop:AonEc} states that the semigroup governing the linear deterministic problem is analytic on $E^c$, while Proposition \ref{prop:AonB} claims that the semigroup is strongly continuous on $B$. In Subsection \ref{subsec:mainresults} we first make the necessary assumptions on the reaction terms in \eqref{eq:fjdef} and on the diffusion coefficients in \eqref{eq:hjdefnov}. We then rewrite \eqref{eq:stochnet} in the form of a stochastic abstract Cauchy problem \eqref{eq:SCPn} and prove the main existence and uniqueness result in Theorem \ref{theo:SCPnsolcontB} and a space-time regularity result in Theorem \ref{theo:Holderreg}.

\begin{notat}
The following notations are used throughout the article with $\ell_{\me}>0$. 
\begin{itemize}
	\item By $C[0,\ell_{\me}]$ we denote the space of continuous functions on the (compact) interval $[0,\ell_{\me}]$, and by $C^1[0,\ell_{\me}]$ we denote such functions from $C[0,\ell_{\me}]$ that are continuously differentiable on $[0,\ell_{\me}]$. These are Banach spaces supplied with the usual maximum-norm. For $\lambda\in (0,1)$, we mean by $C^{\lambda}[0,\ell_{\me}]$ the space of $\lambda$-Hölder-continuous functions with the usual norm.
	\item By $L^p(0,\ell_{\me})$, $1\leq p< \infty$ we denote the Banach space of measurable functions for which the $p$-th power of the absolute value is Lebesgue integrable (where functions which agree almost everywhere are identified), supplied with the usual $L^p$-norm. $L^{\infty}(0,\ell_{\me})$ denotes the space of measurable functions which are bounded almost everywhere, supplied with the supremum-norm.
	\item For $k\in \nat$, $1\leq p< \infty$, $W^{k,p}(0,\ell_{\me})$ is defined as the subset of functions $u$ in $L^{p}(0,\ell_{\me})$ such that $u$ and its weak derivatives up to order $k$ have a finite $L^p$-norm. 
	\[\left\|u\right\|^p_{W^{k,p}(0,\ell_{\me})} \coloneqq \sum_{0\leq j\leq k}\left\|u^{(j)}\right\|^p_{L^p(0,\ell_{\me})},\]
	where $u^{(j)}$ denotes the $j$th derivative of $u$. Let $H^{k}(0,\ell_{\me})\coloneqq W^{k,2}(0,\ell_{\me})$.
	\item By $W_0^{k,p}(0,\ell_{\me})$ we mean the closure of the space of infinitely many times differentiable functions having compact support in $(0,\ell_{\me})$ with respect to the $W^{k,p}(0,\ell_{\me})$-norm. We define $H_0^{k}(0,\ell_{\me})\coloneqq W_0^{k,2}(0,\ell_{\me})$.
	\item For $s>0$, $1\leq p<\infty$ we define
	\[W^{s,p}(0,\ell_{\me})\coloneqq \left\{u\in W^{\left\lfloor s \right\rfloor}(0,\ell_{\me})\colon \left|u^{(\left\lfloor s \right\rfloor)}\right|_{\theta,p,(0,\ell_{\me})}<\infty\right\},\]
	where $\lfloor s \rfloor$ denotes the integer part of $s$, $\theta=s-\lfloor s \rfloor$ and
	\[\left|u\right|_{\theta,p,(0,\ell_{\me})}^p \coloneqq \int_{0}^{\ell_{\me}} \int_0^{\ell_{\me}} \frac{|u(x)-u(y)|^p}{|x-y|^{1+\theta \cdot p}}\dx\dy.\]
	Furthermore, 
	\[\left\|u\right\|_{W^{s,p}(0,\ell_{\me})} \coloneqq \|u\|_{W^{\left\lfloor s \right\rfloor,p}(0,\ell_{\me})}+\left|u^{(\left\lfloor s \right\rfloor)}\right|_{\theta,p,(0,\ell_{\me})}.\]
	$H^{s}(0,\ell_{\me})\coloneqq W^{s,2}(0,\ell_{\me})$.
	\item Similarly as above we can define $W_0^{s,p}(0,\ell_{\me})$ for $s>0$ and $1\leq p<\infty$ as the closure of the space of infinitely many times differentiable functions having compact support in $(0,\ell_{\me})$ with respect to the $W^{s,p}(0,\ell_{\me})$-norm, and $H_0^{s}(0,\ell_{\me})\coloneqq W^{s,2}(0,\ell_{\me})$.
\end{itemize}
\end{notat}

\section{Well-posedness of the deterministic problem}\label{sec:determnetwork}


We start with the deterministic problem
\begin{equation}\label{eq:netcp}
\left\{\begin{aligned}
\dot{u}_{\me}(t,x) & =  (c_{\me} u_{\me}')'(t,x)-p_{\me}(x) u_{\me}(t,x), &x\in (0,\ell_{\me}),\;t>0,\; \me\in\mE,\;\; & (a)\\
0 & = I_{\mv}U(t,\mv),\;   &t>0,\; \mv\in\mV,\;\;& (b)\\
0 & = M L u(t)+ C u(t), & t>0,\;\;& (c)\\
u_{\me}(0,x) & =  u_{0,\me}(x), &x\in [0,\ell_{\me}],\;\; \me\in\mE,\;\;& (d)
\end{aligned}
\right.
\end{equation}
see also  \cite{Mu07}, \cite{MR07}. 

In the rest of the paper we set the following assumptions on the matrix $M$.
\begin{asum}\label{as:M}
The matrix $M=\left(b_{ik}\right)_{n\times n}$ is 
\begin{enumerate}
	\item real, symmetric,
	\item for $i \neq k,$ $b_{ik}\geq 0$, that is, $M$ has positive off-diagonal;
	\item 
\[b_{ii}+\sum_{k\neq i}b_{ik}< 0, \quad i=1,\dots ,n.\]
\end{enumerate}
\end{asum}
We would like to rewrite our system in the form of an abstract Cauchy problem. First we consider the Hilbert space
\begin{equation}\label{eq:E2}
E_2\coloneqq \prod_{\me\in\mE}L^2\left(0,\ell_{\me}\right)
\end{equation}
as the \emph{state space} of the edges, endowed with the natural inner product
\[\langle u,v\rangle_{E_2}\coloneqq \sum_{\me\in\mE} \int_0^{\ell_{\me}} u_{\me}(x)v_{\me}(x) dx,\qquad
u, v \in E_2.\]
On $E_2$ we define the operator
\begin{equation}\label{eq:opAmax}
A_{\max}\coloneqq \diag\left(\frac{d}{dx}\left(c_{\me} \frac{d}{dx}\right)-p_{\me} \right)_{\me\in\mE}
\end{equation}
with maximal domain
\begin{equation}\label{eq:domAmax}
D(A_{\max})\coloneqq \left\{u\in\prod_{\me\in\mE} H^2(0,\ell_{\me}): 0 = I_{\mv}U(\mv),\; \mv\in\mV\right\}
\end{equation}
containing only the continuity condition in its domain.

With these notations, we can rewrite \eqref{eq:netcp} in form of an abstract Cauchy problem in the same way as in \cite[(2.8)]{KS21}. Define
\begin{align}\label{eq:amain}
A_2&\coloneqq A_{\max}\\
D(A_2)&\coloneqq\left\{u\in D(A_{\max})\colon 0=MLu + Cu\right\}.
\end{align}
Using this, \eqref{eq:netcp} becomes
\begin{equation}\label{eq:acp}
\begin{split}
\dot{u}(t)&= A_2 u(t),\; t> 0,\\
u(0)&=u_{0},
\end{split}
\end{equation}
with $u_0=(u_{0,\me})_{\me\in\mE}$.

In what follows the notion of semigroup and its generator is understood in the sense of \cite[Def.~3.2.5]{ABHN11}. That is a strongly continuous function $T \colon (0,\infty)\to \mathcal{L}(E)$, (where $E$ is a Banach space and $\mathcal{L}(E)$ denotes the bounded linear operators on $E$) satisfying
\begin{enumerate}[(a)]
	\item $T(t + s) = T(t)T(s)$, $s,\, t > 0$,
	\item there exists $c > 0$ such that $\|T(t)\| \leq c$ for all $t\in (0, 1]$,
	\item $T(t)x = 0$ for all $t > 0$ implies $x = 0$
\end{enumerate}
is called a \emph{semigroup}. By the proof of \cite[Thm.~3.1.7]{ABHN11} there exist constants $M,\omega\geq 0$ such that $\|T(t)\|\leq M\e^{\omega t}$ for all $t>0$. From \cite[Prop.~3.2.4]{ABHN11} we obtain that there exists a unique operator $A$ with $(\omega ,\infty ) \subset \rho(A)$ and
\begin{equation}\label{eq:}
R(\la,A)=\int_0^{\infty}\e^{-\la t}T(t)\dt\quad (\la>\omega),
\end{equation}
and we call $(A,D(A))$ the \emph{generator} of $T$. The semigroup $T$ is \emph{strongly continuous} or $C_0$, that is $T \colon [0,\infty)\to \mathcal{L}(E)$ is a strongly continuous function satisfying the property
\begin{align}
T(t + s) &= T(t)T(s),\quad  s,\, t \geq 0,\\
T(0)&=Id
\end{align}
if and only if its generator is densely defined, see \cite[Cor.~3.3.11.]{ABHN11}. According to \cite[Def.~3.7.1]{ABHN11}, we call the semigroup \emph{analytic} or \emph{holomorphic} if there exists $\theta \in (0,\frac{\pi}{2}]$ such that $T$ has a holomorphic extension to the sector
\[\Sigma_{\theta}\coloneqq \{z\in\comp\setminus \{0\}\colon |\mathrm{arg}\, z|<\theta\}\]
which is bounded on $\Sigma_{\theta'}\cap \{z\in\comp\colon |z|\leq 1\}$ for all $\theta'\in (0,\theta)$. We say that an analytic semigroup is \emph{contractive} when the semigroup operators considered on the positive real half-axis are contractions. We call a semigroup \emph{bounded analytic} or \emph{bounded homomorphic} of angle $\theta \in (0,\frac{\pi}{2}]$ if it has a bounded holomorphic extension to $\Sigma_{\theta'}$ for all $\theta'\in (0,\theta)$.

Furthermore, we say that the operator $(A,D(A))$ is \emph{sectorial} if there exists $\theta \in (0,\frac{\pi}{2}]$ such that
\[\Sigma_{\frac{\pi}{2}+\theta}\subset \rho(A)\text{ and }\sup_{\lambda\in \Sigma_{\frac{\pi}{2}+\theta-\ve}}\|\lambda R(\lambda,A) \|<\infty\text{ for all }\ve>0.\]
By \cite[Thm.~3.7.11]{ABHN11} we know that the generators of bounded analytic semigroups are exactly the sectorial operators.

In the same way as in \cite{KS21} we can prove the following important result.

\begin{prop}\label{prop:A2tul}
The operator $(-A_2,D(A_2))$ in \eqref{eq:amain}, is the operator associated with the form
\begin{equation}\label{eq:ea}
\begin{split}
D(\ea)&=\left\{u\in \prod_{\me\in\mE}H^1(0,\ell_{\me})\colon I_{\mv}U(\mv)=0,\; \mv\in\mV\right\},\\
\ea(u,v)&=\sum_{\me\in\mE}\int_0^{\ell_{\me}}c_{\me}\cdot u_{\me}'\cdot v_{\me}'+\sum_{\me\in\mE}\int_0^{\ell_{\me}} p_{\me}\cdot u_{\me}\cdot v_{\me}-\langle MLu,Lv\rangle_{\real^n}
\end{split}
\end{equation}
in the following sense: 
\begin{equation}
\begin{split}
D(A_2)&=\left\{u\in D(\ea) \colon \exists\, h\in E_2\text{ s.t. }\ea(u,v)=\langle h,v\rangle_{E_2}\text{ for all }v\in D(\ea)\right\},\\
-A_2 u&=h.
\end{split}
\end{equation}
The form $(\ea,D(\ea))$ is symmetric, densely defined, continuous, closed and accretive. The operator $(A_2,D(A_2))$ is densely defined, dissipative, sectorial and self-adjoint with $(0,+\infty)\subset\rho(A_2)$.
The strongly continuous semigroup $(T_2(t))_{t\geq 0}$ generated by $(A_2,D(A_2))$ is bonded analytic, positive and contractive. 
\end{prop}
\begin{proof}
We can apply \cite[Lem.~3.4]{Mu07} and \cite[Lem.~3.3]{MR07} directly to obtain the first part of the statement. The properties of $\ea$ follow by the proofs of \cite[Prop.~3.2 and Cor.~3.3]{Mu07} and \cite[Lem.~3.2]{MR07}, using Assumption \ref{as:M} and \eqref{eq:pe}. The statements \cite[Prop.~1.51 and Thm.~1.52]{Ou04} imply the properties of the semigroup $(T_2(t))_{t\geq 0}$ generated by $(A_2,D(A_2))$. The self-adjointness of $A_2$ is true by \cite[Prop.~1.24]{Ou04}, the remaining properties follow by standard semigroup theory, see e.g.~\cite[Thm.~3.4.5 and Thm.~3.7.11]{ABHN11}.
\end{proof}

For our results regarding \eqref{eq:stochnet} we will need well-posedness of \eqref{eq:acp} not only on the Hilbert space $E_2$ but also on $L^p$-spaces. Therefore, we introduce the following notions. Let
\begin{equation}\label{eq:Ep}
E_p\coloneqq \prod_{\me\in\mE}L^p\left(0,\ell_{\me}\right),\quad p\in [1,\infty],
\end{equation}
endowed with the norm
\begin{equation}
\begin{split}
\left\|u\right\|_{E_p}^p&\coloneqq \sum_{\me\in\mE}\left\|u_{\me}\right\|_{L^p\left(0,\ell_{\me}\right)}^p,\quad u\in E_p,\; p\in [1,\infty)\\
\left\|u\right\|_{E_{\infty}}&\coloneqq \max_{\me\in\mE}\left\|u_{\me}\right\|_{L^{\infty}\left(0,\ell_{\me}\right)},\quad u\in E_{\infty}.
\end{split}
\end{equation}

\begin{prop}\label{prop:sgrextend}
The semigroup $(T_2(t))_{t\geq 0}$ extends to a family of analytic, contractive, positive one-parameter semigroups $(T_p(t))_{t\geq 0}$ on $E_p$ for $1\leq p\leq \infty$, generated by $(A_p,D(A_p))$. These semigroups are strongly continuous if $p\in[1,\infty)$ and consistent in the sense that if $q,p\in [1,\infty]$ and $q\geq p$ then
\begin{equation}\label{eq:sgrconsistent}
T_p(t)u=T_q(t)u\text{ for }u\in E_q.
\end{equation}
\end{prop}
\begin{proof}
The proof can be carried out analogously to that of \cite[Prop.~2.4.2]{KS21JEE}. The main tool for obtaining the desired family of semigroups $(T_p(t))_{t\geq 0}$ is to prove that $(T_2(t))_{t\geq 0}$ is sub-Markovian and has Gaussian upper bound. The previous one can be verified in the same way as it has been done in the proof of \cite[Lem.~B.1]{KS21JEE}, using Assumption \ref{as:M} and the properties of the form $\ea$ from Proposition \ref{prop:A2tul}. For the existence of Gaussian upper bound we use that by Assumption \ref{as:M}, $-M$ is positive definite, and thus we can apply \cite[Thm.~4.1]{PZ11} directly.
\end{proof}

We also can prove that the generators $(A_p,D(A_p))$ for $p\in [2,+\infty)$ have in fact the same form as $A_2$ on $E_2$, with appropriate domain. Let
\begin{equation}\label{eq:opApmax}
\begin{split}
A_{p,\max}& \coloneqq \diag\left(\frac{d}{dx}\left(c_{\me} \frac{d}{dx}\right)-p_{\me} \right)_{\me\in\mE}\\
D(A_{p,\max})&\coloneqq \left\{u\in\prod_{\me\in\mE} W^{2,p}(0,\ell_{\me}): 0 = I_{\mv}U(\mv),\; \mv\in\mV\right\}.
\end{split}
\end{equation}
\begin{lemma}\label{lem:Ap}
For all $p\in [2,\infty)$ the semigroup generators $(A_p,D(A_p))$ are given by
\begin{equation}\label{eq:Ap}
\begin{split}
A_p& =A_{p,\max},\\
D(A_p)&=\left\{ u\in D(A_{p,\max})\colon 0=MLu + Cu\right\}.
\end{split}
\end{equation}
\end{lemma}
\begin{proof}
For $p=2$ the statement follows by \eqref{eq:amain}. Let $2<p<+\infty$. By \eqref{eq:sgrconsistent} we know that $(T_2(t))_{t\geq 0}$ leaves $E_p$ invariant and $T_p(t)=T_2(t)|_{E_p}$. Since $E_p\hookrightarrow E_2$ holds and the semigroups are strongly continuous, we can apply \cite[Prop.~II.2.3]{EN00} and obtain that $A_p$ is the part of $A_2$ in $E_p$. Hence, a direct computation yields  \eqref{eq:Ap}.
\end{proof}

\section{A stochastic reaction-diffusion equation on a metric graph}\label{sec:srden}

\subsection{An abstract result for a stochastic reaction-diffusion equation}\label{sec:srdenabstr}

Let $(\Omega,\mathscr{F},\mathbb{P})$ is a complete probability space endowed with a right continuous filtration $\mathbb{F}=(\mathscr{F}_t)_{t\in [0,T]}$ for a given $T>0$. Let $(W_H(t))_{t\in [0,T]}$ be a cylindrical Wiener process, defined on $(\Omega,\mathscr{F},\mathbb{P})$, in some Hilbert space $H$ with respect to the filtration $\mathbb{F}$; that is,
  $(W_H(t))_{t\in [0,T]}$ is $(\mathscr{F}_t)_{t\in [0,T]}$-adapted and for all $t>s$, $W_H(t)-W_H(s)$ is independent of $\mathscr{F}_s$.

We will cite a result from \cite{KS21JEE}, concerning the following abstract equation
\begin{equation}\tag{SCP}\label{eq:SCP}
\left\{
\begin{aligned}
d X(t)&=[A X(t)+F(t,X(t))+\widetilde{F}(t,X(t))]dt+G(t,X(t))d W_{H}(t)\\
X(0)&=\xi.
\end{aligned}
\right.
\end{equation}

In what follows let $E$ be a real Banach space. Occasionally -- without being stressed -- we have to pass to appropriate complexification (see e.g. \cite{MST99}) when we use sectoriality arguments.

If we assume that $(A, D(A))$ generates a strongly continuous, analytic semigroup $S$ on the Banach space $E$ with $\Vert S(t)\Vert\leq M\e^{\omega t}$, $t\geq 0$ for some $M\geq 1$ and $\omega\in\real$, then for $\omega'>\omega$ the fractional powers $(\omega'-A)^{\alpha}$ are well-defined for all $\alpha\in(0,1).$ In particular, the fractional domain spaces
\begin{equation}\label{eq:fractdom}
E^{\alpha}\coloneqq D((\omega'-A)^{\alpha}),\quad \|u\|_{\alpha}\coloneqq \|(\omega'-A)^{\alpha}u\|,\quad u\in D((\omega'-A)^{\alpha})
\end{equation}
are Banach spaces. It is well-known (see e.g. \cite[$\mathsection$II.4--5.]{EN00}), that up to equivalent norms, these space are independent of the choice of $\omega'.$

For $\alpha\in(0,1)$ we define the extrapolation spaces $E^{-\alpha}$ as the completion of $E$ under the norms $\|u\|_{-\alpha}\coloneqq \|(\omega'-A)^{-\alpha}u\|$, $u\in E$. These spaces are independent of $\omega'>\omega$ up to an equivalent norm.

We fix $E^0\coloneqq E$.

\begin{rem}\label{rem:omega0}
If $\omega=0$ (hence, the semigroup $S$ is bounded), then by \cite[Proposition 3.1.7]{Haase06}
we can choose $\omega'=0$. That is,
\[E^{\alpha}\cong D((-A)^{\alpha}),\quad \alpha\in [0,1),\]
when $D((-A)^\alpha)$ is equipped with the graph norm.
\end{rem}
Let $Z$ be a Banach space. For $u\in Z$ we define the \emph{subdifferential of the norm at} $u$ as the set
\begin{equation}\label{eq:subdiff}
\partial\|u\|\coloneqq \left\{u^*\in Z^*\colon \|u^*\|_{Z^*}=1\text{ and }\langle u,u^*\rangle=1\right\}
\end{equation}
which is not empty if $u\neq 0$ by the Hahn-Banach theorem.

We now introduce the following assumptions for the operators in (SCP), see \cite[Ass.~3.7]{KS21JEE}. Here, unless stated otherwise, $\|\cdot \|$ will denote the norm of $Z$.


\begin{assum}$ $\label{assum:main}
\begin{enumerate}
	\item Let $E$ be a UMD Banach space of type $2$ and $(A, D(A))$ a densely defined, closed and sectorial operator on $E$.
	\item We have continuous (but not necessarily dense) embeddings for some $\theta\in (0,1)$ \[E^{\theta}\hookrightarrow Z\hookrightarrow E.\]
	\item The strongly continuous analytic semigroup $S$ generated by $(A, D(A))$ on $E$ restricts to an analytic, contractive semigroup, denoted by $\mcS^{z}$ on $Z$, with generator $(\mcA^z,D(\mcA^z))$.
	\item The map $F\colon [0,T]\times\Omega\times Z\to Z$ is continuous in the first variable and locally Lipschitz continuous in the third variable in the sense that for all $r>0$, there exists a constant $L_{F}^{(r)}$ such that
	      \[\left\|F(t,\omega,u)-F(t,\omega,v)\right\|\leq L_{F}^{(r)}\|u-v\|\]
	      for all $\|u\|,\|v\|\leq r$ and $(t,\omega)\in [0,T]\times \Omega$ and there exists a constant $C_{F,0}\geq 0$ such that
				\[\left\|F(t,\omega,0)\right\|\leq C_{F,0},\quad t\in[0,T],\; \omega\in\Omega.\]
				Moreover, for all $u\in Z$ the map $(t,\omega)\mapsto F(t,\omega,u)$ is strongly measurable and adapted.\\
				Finally, for suitable constants $a',b'\geq 0$ and $N\geq 1$ we have
				\[\langle A u+F (t,\omega,u+v), u^*\rangle\leq a'(1+\|v\|)^N+b'\|u\|\]
				for all $u\in D(A|_{Z})$, $v\in Z$, $\omega\in\Omega$ and $u^*\in\partial\|u\|,$ see \eqref{eq:subdiff}.
	\item There exist constants $a'',\, b'',\, k,\, K>0$ with $K\geq k$ such that the function $F\colon [0,T]\times\Omega\times Z\to Z$ satisfies
	      \begin{equation}\label{eq:assymF}
				\langle F (t,\omega, u+v)-F (t,\omega, v), u^*\rangle\leq a''(1+\|v\|)^K-b''\|u\|^k
				\end{equation}
				for all $t\in[0,T]$, $\omega\in\Omega$, $u,v\in Z$ and $u^*\in\partial\|u\|,$ and
				\[\left\|F(t,\omega,v)\right\|\leq a''(1+\|v\|)^K\]
				for all $v\in Z.$, , $\omega\in\Omega$.
	\item For some constant $\kF\geq 0$, the map $\widetilde{F}\colon [0,T]\times\Omega\times Z\to E^{-\kF}$ is globally Lipschitz continuous in the third variable, uniformly with respect to the first and second variables.
				Moreover, for all $u\in Z$ the map $(t,\omega)\mapsto \widetilde{F}(t,\omega,u)$ is strongly measurable and adapted.\\
				Finally, for some $d'\geq 0$ we have
				\begin{equation}\label{eq:Ftildegrow2}
				\left\|\widetilde{F}(t,\omega,u)\right\|_{E^{-\kF}}\leq d'\left(1+\|u\|\right)
				\end{equation}
				for all $(t,\omega,u)\in [0,T]\times \Omega\times Z.$
	\item Let $\gamma(H,E^{-\kG})$ denote the space of $\gamma$-radonifying operators from the Hilbert space $H$ to $E^{-\kG}$ for some $0\leq \kG<\frac{1}{2}$ (see e.g. \cite[Sec.~3.1]{KvN12}). Then the map $G\colon [0,T]\times\Omega\times Z\to \gamma(H,E^{-\kG})$ is locally Lipschitz continuous in the sense that for all $r>0$, there exists a constant $L_{G}^{(r)}$ such that
	      \[\left\|G(t,\omega,u)-G(t,\omega,v)\right\|_{\gamma(H,E^{-\kG})}\leq L_{G}^{(r)}\|u-v\|\]
	      for all $\|u\|,\|v\|\leq r$ and $(t,\omega)\in [0,T]\times \Omega$. 
				Moreover, for all $u\in Z$ and $h\in H$ the map $(t,\omega)\mapsto G(t,\omega,u)h$ is strongly measurable and adapted.\\
				Finally, for suitable constant $c',$
				\[\left\|G(t,\omega,u)\right\|_{\gamma(H,E^{-\kG})}\leq c'\left(1+\|u\|\right)^{\frac{k}{K}}\] 
				for all $(t,\omega,u)\in [0,T]\times \Omega\times Z.$
\end{enumerate}
\end{assum}

\begin{rem}
In Assumptions \ref{assum:main}(3) we use the fact that since $S$ is analytic on $E$ and by Assumptions \ref{assum:main}(2), $D(A)\subset E^{\theta}\hookrightarrow Z$ holds, $S$ leaves $Z$ invariant. Hence, the restriction $\mcS^z$ of $S$ on $Z$ makes sense, and by assumption, $\mcS^z$ is an analytic contraction semigroup on $Z$. 
Note that since $\mcS^z$ is not necessarily strongly continuous, $\mcA^z$ is not necessarily densely defined. However, one can prove that $(\mcA^z,D(\mcA^z))$ is the \emph{part} of $(A, D(A))$ in $Z$, see \cite[Rem.~3.4]{KS21JEE}. 
\end{rem}

We recall that a mild solution of \eqref{eq:SCP} is a solution of the following integral equation
\begin{align}\label{eq:mildsolSCP}
X(t)&=S(t)\xi+\int_0^tS(t-s)\left(F(s,X(s))+\widetilde{F}(s,X(s))\right)\ds\notag\\
&+\int_0^tS(t-s)G(s,X(s))\,dW_H(s)\notag\\
&\eqqcolon S(t)\xi+S\ast F(\cdot,X(\cdot))(t)+S\ast \widetilde{F}(\cdot,X(\cdot))(t)+S\diamond G(\cdot,X(\cdot))(t)
\end{align}
where
\[S\ast f(t)=\int_0^tS(t-s)f(s)\ds\]
denotes the "usual" convolution, and
\[S\diamond g(t)=\int_0^tS(t-s)g(s)\,dW_H(s)\]
denotes the stochastic convolution with respect to $W_H.$ We also implicitly assume that all the terms on the right hand side of  \eqref{eq:mildsolSCP} are well-defined.
\medskip

We now cite the claim of \cite[Thm.~3.10]{KS21JEE} that will be crucial for the proof of our results. For a fixed $T>0$ and $q\geq 1$, define the space
\begin{equation}\label{eq:V_T}
V_{T,q}\coloneqq L^q\left(\Omega;C((0,T];Z)\cap L^{\infty}(0,T;Z)\right)
\end{equation}
being a Banach space with norm
\begin{equation}\label{eq:V_Tnorm}
\left\|u\right\|^q_{V_{T,q}}\coloneqq \mathbb{E}\sup_{t\in [0,T]}\|u(t)\|^q,\quad u\in V_{T,q}.
\end{equation} 
\begin{theo}[Theorem 3.10 of \cite{KS21JEE}]\label{theo:JEE}
Let $T>0$, $2<q<\infty$ and suppose that Assumptions \ref{assum:main} hold with 
\[\theta+\kF<1,\quad \theta+\kG<\frac{1}{2}-\frac{1}{q}.\]
Then for all $\xi\in L^q(\Omega,\mathscr{F}_0,\mathbb{P}; Z)$ there exists a global mild solution $X\in V_{T,q}$ of \eqref{eq:SCP}.  Moreover, for some constant $C_{q,T}>0$ we have
\[\|X\|_{V_{T,q}}^q\leq C_{q,T}\left(1+\mathbb{E}\|\xi\|^q\right).\]
\end{theo}

\subsection{Preparatory results}\label{sec:srdenprep}


To be able to verify our main results we need to prove some preliminary results regarding the setting of Section \ref{sec:determnetwork}. 
We make use of the fact that the semigroups involved here all leave the corresponding real spaces invariant, see \cite[Sec.~4.1]{KS21JEE}.

\begin{defi}\label{defi:Ec}
We denote by
\begin{equation}\label{eq:Ec}
E^c:=\prod_{\me\in\mE}C[0,\ell_{\me}]
\end{equation}
the product space of continuous functions on the edges (not necessarily continuous in the vertices in the sense of Definition \ref{defi:contG}).
The norm is defined as usual with
\begin{equation}\label{eq:normEc}
\left\|u\right\|_{E^c}\coloneqq \max_{\me\in\mE}\|u_{\me}\|_{C[0,\ell_{\me}]},\quad u\in E^c.
\end{equation}
\end{defi}
This space will play the role of the space $Z$ in our setting.\smallskip

In the rest of the paper we assume that 
\[p\geq 2\]
holds. This is justified by the fact that we will apply Theorem \ref{theo:JEE} and in Assumptions \ref{assum:main}(1), $E$ is assumed to be a UMD Banach space of type $2$.\smallskip

We first need that the semigroups $(T_p(t))_{t\geq 0}$ restrict to the same analytic semigroup of contractions on $E^c$.

\begin{prop}\label{prop:AonEc}
For all $p\in[2,+\infty]$, the semigroups $T_p$ leave $E^c$ invariant, and the restrictions $T_p|_{E^c}$ all coincide that we denote by $\mcS^c$. The semigroup $\mcS^c$ is analytic and contractive on $E^c$. Its generator $(\mcA^c,D(\mcA^c))$ coincides with the part $(A_p|_{E^c},D(A_p|_{E^c}))$ of the operator $(A_p,D(A_p))$ in $E^c$ for any $p\in[2,\infty]$.
\end{prop}
\begin{proof}
First we show that for all $p\in[2,+\infty]$, $D(A_p)\subset E^c$ holds. By Lemma \ref{lem:Ap} and Sobolev embedding we have that $D(A_p)\subset E^c$ holds if $p\in[2,+\infty)$. 
The inclusion $D(A_{\infty})\subset D(A_2)\subset E^c$ and the rest of the proof can be carried out in the same way as in the proof of \cite[Prop.~4.4]{KS21JEE}.
\end{proof}

We remark that the semigroup $\mcS^c$ is not strongly continuous on $E^c$ since $D(\mcA^c)\subset D(A_2)$ and $D(A_2)$ is not dense in $E^c$, see \eqref{eq:domAmax} and \eqref{eq:amain}.

In our main results the following closed subspace of $E^c$ plays a fundamental role.

\begin{defi}\label{defi:B}
We denote by
\begin{equation}\label{eq:B}
B\coloneqq \left\{u\in\prod_{\me\in\mE}C[0,\ell_{\me}]\colon I_{\mv}U(\mv)=0,\;\mv\in\mV\right\}, 
\end{equation} 
the space of continuous functions on $\mG$, see Definition \ref{defi:contG}. The space $B$ is clearly a closed subspace of $E^c$, hence a Banach space with norm
\begin{equation}\label{eq:normB}
\|u\|_{B}=\max_{\me\in\mE}\|u_{\me}\|_{C[0,\ell_{\me}]}=\|u\|_{E^c},\quad u\in B.
\end{equation} 
\end{defi}

\begin{prop}\label{prop:AonB}
The part of the generator $\left(A_p, D(A_p)\right)$ in $B$ is the same operator for each $p\in [2,+\infty]$ and this operator is the generator of a positive, strongly continuous, analytic contraction semigroup on $B$.
\end{prop}
\begin{proof}
In an analogous way as in the proof of Proposition \ref{prop:AonEc} we can verify that $D(A_p)\subset B$ holds for $p\in [2,+\infty]$. Using the analiticity of the semigroups $(T_p(t))_{t\geq 0}$ we obtain that the space $B$ is $T_p(t)$-invariant. Applying the properties of the form $\ea$ from Proposition \ref{prop:A2tul} and the analyticity and compatibility of the semigroups $\left(T_p(t)\right)_{t\geq 0}$ on the spaces $E_p$ (see Proposition \ref{prop:sgrextend}) we obtain in the same way as in the proof of \cite[Prop.~2.1]{KS21Cor} that the part of the generator $\left(A_p, D(A_p)\right)$ in $B$ is the same operator for each $p\in [2,+\infty]$ and it generates a strongly continuous contraction semigroup on $B$.  The analyticity follows by observing that $B$ is a closed subspace of $E^c$, and $\mcS^c$ is analytic on $E^c$.
\end{proof}

We recall that for $p\in [2,+\infty)$ the operators $(A_p,D(A_p))$ are generators of strongly continuous analytic semigroups on the spaces $E_p$ defined in \eqref{eq:Ep}  (see Proposition \ref{prop:sgrextend}).
\begin{equation}\label{eq:fractdomEp}
\text{For }0 \leq \theta< 1\text{ let }E_p^{\theta}\text{ be defined as in \eqref{eq:fractdom} for the operator }A_p\text{ on the space }E_p. 
\end{equation}

We will need the following embedding result regarding the fractional domain spaces and the space $B$.

\begin{prop}\label{prop:fractionalspaceinclB}
Let $E_p^{\theta}$ defined in \eqref{eq:fractdomEp}. Then for $p\in[2,\infty)$ and $\theta> \frac{1}{2p}$ the following continuous, dense embeddings are satisfied:
\begin{equation}\label{eq:EthetainclB}
E_p^{\theta}\hookrightarrow B\hookrightarrow E_p.
\end{equation} 
\end{prop}
\begin{proof}
In an analogous way as done in \cite[Lem.~3.6]{KS21}, one can verify that
\begin{equation}\label{eq:Biso}
B\cong \left(\prod_{\me\in\mE}C_0[0,\ell_{\me}]\right)\times\real^n,
\end{equation}
where $C_0[0,\ell_{\me}]$ denotes the space of continuous functions on $[0,\ell_{\me}]$ that have a value of $0$ at the endponints of the interval.
By Proposition \ref{prop:sgrextend} the operator $(A_p,D(A_p))$ generates a positive, contraction semigroup on $E_p$.
It follows from \cite[Thm.~in $\mathsection$4.7.3]{Ar04} and \cite[Prop.~in $\mathsection$4.4.10]{Ar04} that for the complex interpolation spaces
\begin{equation}\label{eq:inclpf2}
D((\la-A_p)^{\theta})\cong [D(\la-A_p),E_p]_{\theta}
\end{equation}
holds for any $\la>0$. It is straightforward from \cite[Lem.~3.5]{KS21} that we have 
\begin{equation}\label{eq:Apmaxiso}
D(A_{p,\max})\cong \left(\prod_{\me\in\mE}W_0^{2,p}(0,\ell_{\me})\right)\times\real^n,
\end{equation}
see \eqref{eq:opApmax}. Observe that by Lemma \ref{lem:Ap}, for $p\geq 2$, $D(A_p)\hookrightarrow D(A_{p,\max})$ holds. Hence, from \eqref{eq:inclpf2}, \eqref{eq:Apmaxiso} and general interpolation theory, see e.g. \cite[Section 4.3.3]{Triebel78}, we obtain in the same way as in the proof of \cite[Cor.~3.1]{KS21Cor} that for $\theta>\frac{1}{2p}$
\begin{equation}\label{eq:EthetainclW2p}
E_p^{\theta}=D((\la-A_p)^{\theta}) \hookrightarrow \left(\prod_{\me\in\mE} W_{0}^{2\theta,p}(0,\ell_{\me})\right)\times\real^n
\end{equation}
is satisfied. Hence, by \eqref{eq:Biso} and Sobolev embedding we have that
\begin{equation}\label{eq:W02pinclB}
\left(\prod_{\me\in\mE} W_{0}^{2\theta,p}(0,\ell_{\me})\right)\times\real^n \hookrightarrow B,
\end{equation}
thus
\begin{equation}
E_p^{\theta} \hookrightarrow B
\end{equation}
holds with continuous embedding. Proposition \ref{prop:AonB} implies that $D(A_p)$, hence $E_p^{\theta}\cong D((-A_p)^{\theta})$ is dense in $B$. Since $\prod_{\me\in\mE}C_0[0,\ell_{\me}]\subset B$ is a dense subspace of $E_p$, we obtain that $B$ is continuously and densely embedded in $E_p$.
\end{proof}

As a corollary, we obtain the following chain of embeddings containing the space $E^c$. It is important to remark that the embeddings are continuous but the first one is not dense.

\begin{cor}\label{cor:fractionalspaceincl}
For $p\in[2,\infty)$ and $\theta> \frac{1}{2p}$ the following continuous embeddings are satisfied:
\begin{equation}\label{eq:Ethetaincl}
E_p^{\theta}\hookrightarrow E^c\hookrightarrow E_p.
\end{equation} 
\end{cor}
\begin{proof}
Since $B\subset E^c$ a (closed) subspace by definition and \eqref{eq:EthetainclB} holds,
\[E_p^{\theta}\hookrightarrow E^c\]
is satisfied as a continuous embedding. The claim follows by observing $E^c\hookrightarrow E_p$ with continuous, dense embedding.
\end{proof}

\subsection{Main results}\label{subsec:mainresults}

We now apply the results of the previous sections to the following stochastic evolution equation, based on \eqref{eq:netcp}. We prescribe stochastic noise on the edges of the network and add a nonlinear and a first order term to the first equation of \eqref{eq:netcp}.

Let $(\Omega,\mathscr{F},\mathbb{P})$ be a complete probability space endowed with a right-continuous filtration $\mathbb{F}=(\mathscr{F}_t)_{t\in [0,T]}$ for some $T>0$ given. We consider the problem
\begin{equation}\label{eq:stochsys}
\left\{\begin{aligned}
\dot{u}_{\me}(t,x) & =  (c_{\me} u_{\me}')'(t,x)+d_{\me}(x)\cdot u_{\me}'(t,x)&&\\
&-p_{\me}(x) u_{\me}(t,x)+f_{\me}(t,x,u_{\me}(t,x))&&\\
&+ h_{\me}(t,x,u_{\me}(t,x))\frac{\partial w_{\me}}{\partial t}(t,x), &x\in (0,\ell_{\me}),\; t\in(0,T],\; \me\in\mE,\;\; & (a)\\
0 & = I_{\mv}U(t,\mv),\;   &t\in(0,T],\; \mv\in\mV,\;\;& (b)\\
0 & = M L u(t)+ C u(t), & t\in(0,T],\;\;& (c)\\
u_{\me}(0,x) & =  u_{0,\me}(x), &x\in [0,\ell_{\me}],\; \me\in\mE.\;\;& (d)
\end{aligned}
\right.
\end{equation}

The terms $\frac{\partial w_{\me}}{\partial t}$ are independent space-time white noises on $[0,\ell_{\me}]$, $\me\in\mE$, written as formal derivatives of independent cylindrical Wiener-processes $(w_{\me}(t))_{t\in [0,T]}$, defined on $(\Omega,\mathscr{F},\mathbb{P})$, in the Hilbert space $L^2(0,\ell_{\me})$  with respect to the filtration $\mathbb{F}$.\smallskip

\begin{assum}\label{assum:fj}
The functions $f_{\me}\colon [0,T]\times \Omega\times [0,\ell_{\me}]\times \real\to\real$ are polynomials of the form
\begin{equation}\label{eq:fjdef}
f_{\me}(t,\omega,x,\eta)=-a_{\me}(t,\omega,x)\eta^{2k_{\me}+1}+\sum_{j=0}^{2k_{\me}}a_{{\me},j}(t,\omega,x)\eta^j,\quad \eta\in\real,\, {\me}\in\mE
\end{equation} 
for some positive integers $k_{\me},$ ${\me}\in\mE$.
For the coefficients we assume that there are constants $0<c\leq C<\infty$ such that
\begin{equation}\label{eq:assa_j}
c\leq a_{\me}(t,\omega,x)\leq C,\;\left|a_{{\me},j}(t,\omega,x)\right|\leq C,\text{ for each }{\me}\in\mE,\;j=0,\dots ,2k_{\me},
\end{equation}
for all $(t,\omega,x)\in [0,T]\times \Omega\times [0,\ell_{\me}]$. Furthermore, we suppose that
\[a_{\me}(t,\omega,\cdot),\,a_{{\me},j}(t,\omega,\cdot)\in C[0,\ell_{\me}],\quad {\me}\in\mE,\;j=0,\dots ,2k_{\me}\]
and that the coefficients $a_{\me},\,a_{{\me},j}\colon [0,T]\times \Omega\times [0,\ell_{\me}]\to\real$ are jointly measurable and adapted in the sense that for each ${\me}$ and $j$ and for each $t\in[0,T]$, the functions $a_{\me}(t,\cdot),\,a_{{\me},j}(t,\cdot)$ are $\mathscr{F}_t\otimes \mcB_{[0,\ell_{\me}]}$-measurable, where $\mcB_{[0,\ell_{\me}]}$ denotes the sigma-algebra of the Borel sets on $[0,\ell_{\me}].$
\end{assum}

\begin{rem}
In \cite[Ass.~3.10]{KS21} we had to assume a much more restrictive condition on the polynomials in \eqref{eq:fjdef}, namely that $k_{\me}=k$ for all $\me\in\mE$ and
\[\left(a_{\me}(t,\omega,\cdot)\right)_{\me\in\mE},\; \left(a_{\me,j}(t,\omega,\cdot)\right)_{\me\in\mE}\in B\]
for $j=1,\dots ,2k$, $t\in[0,T]$ and almost all $\omega\in\Omega$ holds. In our case these assumptions can be omitted.
\end{rem}

Define now
\begin{align}\label{eq:defkK}
K\coloneqq 2k_{\max}+1,\quad k\coloneqq 2k_{\min}+1,\\
\text{ where }k_{\max}=\max_{{\me}\in\mE}k_{\me},\quad k_{\min}=\min_{{\me}\in\mE}k_{\me}.\notag
\end{align}

We suppose that
\begin{align}
&h_{\me}\colon [0,T]\times \Omega\times [0,\ell_{\me}]\times \real\to\real,\quad {\me}\in\mE \text{ are locally Lipschitz continuous}\notag\\
&\text{in the fourth variable, uniformly with respect to the first 3 variables,}\label{eq:hjdefLip}\\
&|h_{\me}(t,\omega, x,\eta)|\leq c(1+|\eta|)^{\frac{k}{K}}\text{ for all }(t,\omega, x,\eta)\in [0,T]\times\Omega\times [0,\ell_{\me}]\times \real.\label{eq:hjdefnov}
\end{align}
We further assume that the functions $h_{\me}$ are jointly measurable and adapted in the sense that for each ${\me}$ and $t\in[0,T]$, $h_{\me}(t,\cdot)$ is $\mathscr{F}_t\otimes \mcB_{[0,\ell_{\me}]}\otimes \mcB_{\real}$-measurable, where $\mcB_{[0,\ell_{\me}]}$ and $\mcB_{\real}$ denote the sigma-algebras of the Borel sets on $[0,\ell_{\me}]$ and $\real$, respectively.\smallskip

We now rewrite system \eqref{eq:stochsys} in the form of an abstract stochastic Cauchy-problem, see \cite[p.~29]{KS21JEE}.

\begin{equation}\tag{SCPn}\label{eq:SCPn}
\left\{
\begin{aligned}
d X(t)&=[A X(t)+F(t,X(t))+\widetilde{F}(t,X(t))]dt+G(t,X(t))dW(t)\\
X(0)&=\xi.
\end{aligned}
\right.
\end{equation}

The operator $(A, D(A))$ is $(A_p, D(A_p))$ for some large $p\in [2,\infty)$, where $p$ will be chosen 
later. Hence, by Proposition \ref{prop:sgrextend}, $A$ is the generator of the strongly continuous analytic semigroup on the Banach space $E_p$, and $E_p$ is a UMD space of type $2$, see \cite[Sec.~6]{AK16}. 

The function $F\colon [0,T]\times \Omega\times E^c\to E^c$ is defined as
\begin{equation}\label{eq:Fdef}
F(t,\omega,u)\coloneqq\left(f_{\me}(t,\omega,\cdot,u_{\me}(\cdot))\right)_{\me\in\mE},\quad t\in[0,T],\;\omega\in\Omega,\; u=\left(u_{\me}\right)_{\me\in\mE}\in E^c.
\end{equation}

Let
\begin{equation}\label{eq:Ftilde}
\begin{split}
\widetilde{F}&\colon\prod_{\me\in\mE}C^1[0,\ell_{\me}]\to E_p,\quad p>1,\\
\widetilde{F} u &\coloneqq \left(d_{\me}\cdot u'_{\me}\right)_{\me\in\mE}.
\end{split}
\end{equation}

To define the operator $G$ we argue in analogy with \cite[p.~29--30.]{KS21JEE}, but we have to change the argumentation at one point because in \eqref{eq:EthetainclW2p} we have only (continuous) embedding but no isomorphism. We first define the multiplication operator $\Gamma\colon [0,T]\times \Omega\times E^c\to\mathcal{L}(E_2)$ as
\begin{equation}\label{eq:Gammadef}
\Gamma(t,\omega,u)y\coloneqq \diag \left(h_{\me}(t,\omega, \cdot,u_{\me}(\cdot)) \right)_{\me\in\mE}\cdot y,\quad u\in E^c,\; y\in E_2.
\end{equation}
Because of the assumptions on the $h_{\me}$'s, $\Gamma$ clearly maps into $\mathcal{L}(E_2).$

Let $(A_2,D(A_2))$ be the generator on $E_2,$ see Proposition \ref{prop:sgrextend}, and pick $\kG\in(\frac{1}{4},\frac{1}{2})$.
By \eqref{eq:EthetainclW2p}, we have that there is a continuous embedding
\begin{equation}\label{eq:imath}
\imath\colon E_2^{\kG} \to \left(\prod_{\me\in\mE} H_{0}^{2\kG}(0,\ell_{\me})\right)\times\real^n\eqqcolon H_1.
\end{equation}
Using \eqref{eq:W02pinclB} and \eqref{eq:EthetainclB}, we obtain that for $p\geq 2$ arbitrary there exists a continuous embedding
\begin{equation}\label{eq:jmath}
\jmath\colon H_1 \to E_p.
\end{equation}
Let $\nu>0$ arbitrary and define now $G$ by
\begin{equation}\label{eq:G2def}
(\nu- A_p)^{-\kG} G(t,\omega, u)y\coloneqq \jmath\, \imath\,  (\nu- A_2)^{-\kG}\Gamma(t,\omega,u)y,\quad u\in E^c,\; y\in E_2.
\end{equation}

\begin{lemma}\label{lem:G2prop}
\hspace{2em}
\begin{enumerate}[1.]
	\item Let $p>1$ arbitrary. Then the mapping defined in \eqref{eq:Ftilde} can be extended to a linear and continuous operator from $E^c$ into $E_p^{-\frac{1}{2}}$, that we also call $\widetilde{F}$.
	\item Let $p\geq 2$ and $\kG\in(\frac{1}{4},\frac{1}{2})$ be arbitrary. The operator $G$ defined in \eqref{eq:G2def} maps $[0,T]\times\Omega\times E^c$ into $\gamma (E_2,E_p^{-\kappa_{G}})$.
\end{enumerate}
\end{lemma}
\begin{proof}
\begin{enumerate}[1.]
	\item Let $p> 1$ arbitrary and define $q\coloneqq (1-\frac{1}{p})^{-1}$. Proceeding in the same way as in the proof of \cite[Lem.~4.6.1]{KS21JEE} we obtain that there exists a continuous extension of $\widetilde{F}$
\begin{equation}
\widetilde{F}_1\colon E^c \to \left(\prod_{\me\in\mE}W_0^{1,q}(0,\ell_{\me})\right)^*.
\end{equation}
By \cite[Thm.~3.1.4]{vNeervenbook} we have
\[\left(E_q^{\frac{1}{2}}\right)^*\cong E_p^{-\frac{1}{2}}.\]
By \eqref{eq:EthetainclW2p}
\begin{equation}
 E_q^{\frac{1}{2}} \hookrightarrow \left(\prod_{\me\in\mE} W_{0}^{1,q}(0,\ell_{\me})\right)\times\real^n.
\end{equation}
holds. Hence there exists a continuous embedding
\begin{equation}
\imath_p\colon \left(\prod_{\me\in\mE} W_{0}^{1,q}(0,\ell_{\me})\right)^*\times\real^n \to \left(E_q^{\frac{1}{2}}\right)^*\cong E_p^{-\frac{1}{2}}.
\end{equation}
Combining all these fact we obtain that
\begin{equation}
\imath_p\widetilde{F}_1\colon E^c \to E_p^{-\frac{1}{2}}
\end{equation}
is a continuous linear extension of $\widetilde{F}$.
\item Using \cite[Lem.~2.1(4)]{vNVW08}, we obtain in a similar way as in \cite[Cor.~2.2]{vNVW08}) that $\jmath \in\gamma (H_1,E_p)$, since $2\kG>\frac{1}{2}$ holds. Hence, by the definition of $G$ and the ideal property of $\gamma$-radonifying operators, the mapping $G$ takes values in $\gamma (E_2,E_p^{-\kG})$.
\end{enumerate}
\end{proof}

The driving noise process $W$ is defined by
\begin{equation}\label{eq:W2def}
W(t)=\left(	w_{\me}(t,\cdot)\right)_{\me\in\mE}
\end{equation}
and thus $(W(t))_{t\in [0,T]}$ is a cylindrical Wiener process, defined on $(\Omega,\mathscr{F},\mathbb{P})$, in the Hilbert space $E_2$ with respect to the filtration $\mathbb{F}$.

Finally, let
\[\xi\coloneqq u_0.\]\smallskip

We recall that a mild solution of \eqref{eq:SCPn} is a solution of the following integral equation
\begin{align}\label{eq:mildsol}
X(t)&=S(t)\xi+\int_0^tS(t-s)\left(F(s,X(s))+\widetilde{F}(s,X(s))\right)\ds\notag\\
&+\int_0^tS(t-s)G(s,X(s))\,dW(s)\notag\\
&\eqqcolon S(t)\xi+S\ast F(\cdot,X(\cdot))(t)+S\ast \widetilde{F}(\cdot,X(\cdot))(t)+S\diamond G(\cdot,X(\cdot))(t)
\end{align}
where
\[S\ast f(t)=\int_0^tS(t-s)f(s)\ds\]
denotes the "usual" convolution, and
\[S\diamond g(t)=\int_0^tS(t-s)g(s)\,dW(s)\]
denotes the stochastic convolution with respect to $W$ and $\left(S(t)\right)_{t\geq 0}$ denotes the strongly continuous semigroup generated by $(A,D(A))$. We also implicitly assume that all the terms on the right hand side of  \eqref{eq:mildsol} are well-defined.
\medskip

To state our the result regarding system \eqref{eq:SCPn}, we define to the analogy of $V_{T,q}$ (see \eqref{eq:V_T}) for a fixed $T>0$ and $q\geq 1$

\begin{equation}\label{eq:V_Thullam}
\widetilde{V}_{T,q}\coloneqq L^q\left(\Omega;C((0,T];B)\cap L^{\infty}(0,T;B)\right)
\end{equation}
being a Banach space with norm
\begin{equation}\label{eq:V_Thullamnorm}
\left\|u\right\|^q_{\widetilde{V}_{T,q}}\coloneqq \mathbb{E}\sup_{t\in [0,T]}\|u(t)\|_{B}^q,\quad u\in \widetilde{V}_{T,q}.
\end{equation}

\begin{theo}\label{theo:SCPnsolcontB}
Let $F$, $\widetilde{F}$, $G$ and $W$ defined as in \eqref{eq:Fdef}, \eqref{eq:Ftilde}, \eqref{eq:G2def} and \eqref{eq:W2def}, respectively. 
Let $q>4$ be arbitrary.
Then for every $\xi\in L^q(\Omega,\mathscr{F}_0,\mathbb{P};E^c)$ equation \eqref{eq:SCPn} has a unique global mild solution in $\widetilde{V}_{T,q}$.
\end{theo}
\begin{proof}
Using the assumption $q > 4$, we can choose $2\leq p<+\infty$, $\theta\in (0,\frac{1}{2})$ and $\kG\in (\frac14 ,\frac12)$ such that
\begin{equation}\label{eq:pthetakappaG}
\theta>\frac{1}{2p}\text{ and } 0 < \theta+ \kG < \frac12 - \frac1q.
\end{equation} 
First we will apply Theorem \ref{theo:JEE} and show that there exists a unique global mild solution in 
\[V^c_{T,q}\coloneqq L^q\left(\Omega;C((0,T];E^c)\cap L^{\infty}(0,T;E^c)\right).\] 
For this purpose we have to show that Assumptions \ref{assum:main} are satisfied with the following cast: $E=E_p$, $H=E_2$, $Z=E^c$, $A=A_p$, the operators defined in \eqref{eq:Fdef}, \eqref{eq:Ftilde}, \eqref{eq:G2def} and \eqref{eq:W2def}, for appropriate $\kF$ and the constants $p$, $\theta$ and $\kG$ chosen in \eqref{eq:pthetakappaG}.
\begin{enumerate}[(a)]
	\item Assumption $(1)$ is satisfied because of the generator property of $A_p$ on $E_p$, see Proposition \ref{prop:sgrextend} and since any $L^p$-space with $p\in[2,\infty)$ has type $2$.
	\item Assumption $(2)$ is satisfied since \eqref{eq:pthetakappaG} holds and we can use Corollary \ref{cor:fractionalspaceincl}.
	\item Assumption $(3)$ is satisfied because of Proposition \ref{prop:AonEc}.
	\item To prove Assumptions $(4)$ and $(5)$ we can proceed in the same way as in the proof of \cite[Thm.~4.7]{KS21JEE} and obtain that $(4)$ holds with $N=K$ and $(5)$ with $K$ and $k$ defined in \eqref{eq:defkK}.
	\item To check Assumption (6) we refer to Lemma \ref{lem:G2prop}. This implies that $\widetilde{F}\colon E^c\to E^{-\kF}$ with $\kF=\frac{1}{2}$. Since $\widetilde{F}$ is a continuous linear operator, the rest of the statement also follows.
	\item To check Assumption (7) we can mimick the proof of \cite[Thm.~4.7]{KS21JEE} in a sligthly modified way. First note that by Lemma \ref{lem:G2prop}, $G$ takes values in $\gamma (E_2,E_p^{-\kG})$ with $\kG$ chosen above. We fix $u,v\in E^c$ and let
\[R\coloneqq \max\left\{\|u\|_{E^c},\|v\|_{E^c}\right\}.\]
Furthermore, we denote the matrix from \eqref{eq:Gammadef} by
\[\mathcal{M}_{\Gamma}(t,\omega,u)\coloneqq \diag \left(h_{\me}(t,\omega, \cdot,u_{\me}(\cdot)) \right)_{\me\in\mE},\quad\text{ for } u\in E^c.\]
For $R>0$ let
\begin{equation}
L_h(R)\coloneqq\max_{\me\in\mE}L_{h_{\me}}(R),
\end{equation}
where the positive constants $L_{h_{\me}}(R)$ are the Lipschitz constants of the functions $h_{\me}$, on the ball of radius $R$, see \eqref{eq:hjdefLip}.
From the right-ideal property of the $\gamma$-radonifying operators and \eqref{eq:G2def} we have that
\begin{align}
&\left\|(-A_p)^{-\kG}\left(G(t,\omega,u)-G(t,\omega,v)\right)\right\|_{\gamma (E_2,E_p)}\\
&\leq \left\|\jmath\, \imath\,   (-A_2)^{-\kG}\right\|_{\gamma (E_2,E_p^{-\kappa_{G}})}\cdot\left\|\Gamma(t,\omega,u)-\Gamma(t,\omega,v)\right\|_{\mathcal{L}(E_2)}\\
&\leq \left\|\jmath\, \imath\,   (-A_2)^{-\kG}\right\|_{\gamma (E_2,E_p^{-\kappa_{G}})}\cdot\left\|\mathcal{M}_\Gamma(t,\omega,u)-\mathcal{M}_\Gamma(t,\omega,v)\right\|_{L^{\infty}\left(\prod_{\me\in\mE}[0,\ell_{\me}],\real^m\right)}\\
&\preceq \left\|\jmath\, \imath\,   (-A_2)^{-\kG}\right\|_{\gamma (E_2,E_p^{-\kappa_{G}})}\cdot L_h(R)\cdot \left\|u-v\right\|_{E^c}.
\end{align}
Hence, we obtain that $G:[0,T]\times E^c\to \gamma (E_2,E_p^{-\kG})$ is locally Lipschitz continuous.\\
Using assumption \eqref{eq:hjdefnov} on the functions $h_{\me}$ and an analogous computation as above, we obtain that $G$ grows as required in Assumption $(7)$ as a map $[0,T]\times E^c\to \gamma (E_2,E_p^{-\kG})$.
\end{enumerate}
Thus, we can apply Theorem \ref{theo:JEE} with $\kF=\frac12$, $\theta$ and $\kG$ having the properties \eqref{eq:pthetakappaG}, and obtain that \eqref{eq:SCPn} has a unique global mild solution in $V^c_{T,q}$.

To prove the claim for $\widetilde{V}_{T,q}$ recall from \eqref{eq:mildsol} that the solution in $V^c_{T,q}$ satisfies the following implicit equation:
\begin{equation}\label{eq:proofmildsol}
X(t)=S(t)\xi+S\ast F(\cdot,X(\cdot))(t)+S\ast \widetilde{F}(X(\cdot))(t)+S\diamond G(\cdot,X(\cdot))(t),
\end{equation}
where $S$ denotes the semigroup generated by $A_p$ on $E_p$. We only have to show that for almost all $\omega\in\Omega$ for the trajectories 
\begin{equation}\label{eq:Holderregpf2}
X(\cdot)\in C((0,T];B)\cap L^{\infty}(0,T;B)
\end{equation} 
holds. Then $X\in \widetilde{\mcV}_{T,q}$ is satisfied since the norms on $E^c$ and $B$ coincide and $X\in V^c_{T,q}$ is true. Using Proposition \ref{prop:fractionalspaceinclB}, Assumptions \ref{assum:main} and techniques from \cite{vNVW08}, we can proceed in the same way as in the proof of \cite[Thm.~4.10]{KS21JEE} to show \eqref{eq:Holderregpf2} by showing it for all the four terms on the right-hand-side of \eqref{eq:proofmildsol}.
\end{proof}

Notice that in the above theorem, for initial values from $E^c$ we obtain trajectories in the (smaller space) $B$.

\begin{rem}\label{rem:SCPnsolcont0TB}
If in Theorem \ref{theo:SCPnsolcontB} the initial condition satisfies $\xi\in L^{q}(\Omega,\mathscr{F}_0,\mathbb{P};B)$ for $q>4$, then the global mild solution belongs even to $L^q(\Omega;C([0,T];B))$ instead of $\widetilde{V}_{T,q}$. This follows from the fact that by Proposition \ref{prop:AonB}, the part of $A_p$ in $B$ generates a strongly continuous semigroup on $B$.
\end{rem}

In the following theorem we will state a result regarding H\"older regularity of the mild solution of \eqref{eq:SCPn}.

\begin{theo}\label{theo:Holderreg}
Let $q > 4$ be arbitrary, $\lambda, \eta>0$ and $p\geq 2$ such that $\lambda+\eta>\frac{1}{2p}$. We assume that $\xi\in L^{Kq}(\Omega;E_p^{\lambda+\eta})$, where $K$ is defined in \eqref{eq:defkK}. If the inequality
\begin{equation}\label{eq:thmHolderreglaeta}
\lambda+\eta<\frac{1}{4}-\frac{1}{q}
\end{equation} 
is fulfilled, then the mild solution $X$ of \eqref{eq:SCPn} from Theorem \ref{theo:SCPnsolcontB} satisfies 
\[X\in L^q(\Omega;C^{\lambda}([0,T],E_p^{\eta})).\]
\end{theo}
\begin{proof}
Using the continuous embedding \eqref{eq:EthetainclB}, we have that
\begin{equation}\label{eq:xiLq}
\xi\in L^{Kq}(\Omega;B)
\end{equation}
holds. Since $Kq>4$, by Remark \ref{rem:SCPnsolcont0TB} there exists a global mild solution
\begin{equation}\label{eq:XinB}
X\in L^{Kq}(\Omega; C([0,T],B)).
\end{equation}
This solution satisfies the following implicit equation (see \eqref{eq:mildsol}):
\begin{equation}\label{eq:proofmildsol2}
X(t)=S(t)\xi+S\ast F(\cdot,X(\cdot))(t)+S\ast \widetilde{F}(\cdot,X(\cdot))(t)+S\diamond G(\cdot,X(\cdot))(t),
\end{equation}
where $S$ denotes the semigroup generated by $A_p$ on $E_p$, $\ast$ denotes the usual convolution, $\diamond$ denotes the stochastic convolution with respect to $W.$ We are going to estimate the $L^q(\Omega;C^{\lambda}([0,T],E_p^{\eta}))$-norm of $X$ using the triangle-inequality in \eqref{eq:proofmildsol2}. 

For the first term we obtain in the same way as in the proof \cite[Thm.~3.14]{KS21}
\begin{equation}\label{eq:Sxi}
\mathbb{E}\|S(\cdot)\xi\|_{C^{\lambda}([0,T],E_p^{\eta})}^q \leq c\cdot \mathbb{E}\|(-A_p)^{\lambda+\eta}\xi\|^q_{E_p}<\infty
\end{equation}
by assumption.

To estimate the $q$th power of the second term
\[\mathbb{E}\|S\ast F(\cdot,X(\cdot))\|_{C^{\lambda}([0,T],E_p^{\eta})}^q\]
we choose $\theta>\frac{1}{2p}$ such that
\[\lambda+\eta+\theta<1-\frac{1}{q}.\]
Using \cite[Lem.~3.6]{vNVW08} with this $\theta$, $\alpha=1$, and $q$ instead of $p$, and obtain that there exist constants $C\geq 0$ and $\ve>0$ such that
\begin{equation}\label{eq:est2nd1}
\|S\ast F(\cdot,X(\cdot))\|_{C^{\lambda}([0,T],E_p^{\eta})}\leq CT^{\ve}\|F(\cdot,X(\cdot))\|_{L^q(0,T;E_p^{-\theta})}.
\end{equation}
We have to estimate the expectation of the $q$th power on the right-hand-side of \eqref{eq:est2nd1}. By Corollary \ref{cor:fractionalspaceincl} we obtain
\[E^c\hookrightarrow E_p\hookrightarrow E_p^{-\theta},\]
since $\theta>\frac{1}{2p}$ holds and $(\omega'-A_p)^{-\theta}$ is an isomorphism between $E_p^{-\theta}$ and $E_p$ for any $\omega'>0$. Using this and Assumptions \ref{assum:main}(5) with $K$ (which holds by the proof of Theorem \ref{theo:SCPnsolcontB}), we have
\begin{align}
\mathbb{E}\|F(\cdot,X(\cdot))\|^q_{L^q(0,T;E_p^{-\theta})}&=\mathbb{E}\int_0^T\|F(s,X(s))\|^q_{E_p^{-\theta}}\ds\\
& \lesssim \mathbb{E}\int_0^T\|F(s,X(s))\|^q_{E^c}\ds\\
& \lesssim \mathbb{E}\int_0^T(1+\|X(s)\|^{Kq}_{E^c})\ds\\
& \lesssim 1+\mathbb{E}\sup_{t\in[0,T]}\|X(t)\|^{Kq}_{B},
\end{align}
where $\lesssim$ denotes that the expression on the left-hand-side is less or equal to a constant times the expression on the right-hand-side.
In the last inequality we could write the $B$-norm of $X(t)$ since by \eqref{eq:XinB}, we have $X(t)\in B$ and the $E^c$-norm coincides with the $B$-norm on $B$, see \eqref{eq:normB}.

This implies that for each $T>0$ there exists $C_T>0$ such that
\begin{equation}\label{eq:SastF}
\left(\mathbb{E}\|S\ast F(\cdot,X(\cdot))\|_{C^{\lambda}([0,T],E_p^{\eta})}^q\right)^{\frac{1}{q}}
\leq C_T\cdot \left(1+\|X(t)\|_{L^{Kq}(\Omega; C([0,T],B))}^{K}\right),
\end{equation}
and the right-hand-side is finite by \eqref{eq:XinB}.

To estimate the $q$th power of the third term
\[\mathbb{E}\|S\ast \widetilde{F}(\cdot,X(\cdot))\|_{C^{\lambda}([0,T],E_p^{\eta})}^q\]
we use that by assumption, $\frac{1}{2}>\frac{1}{2p}$ and
\[\lambda+\eta+\frac{1}{2}<1-\frac{1}{q}\]
hold. Applying \cite[Lem.~3.6]{vNVW08} with $\theta=\frac{1}{2}$, $\alpha=1$, and $q$ instead of $p$, and obtain that there exist constants $C\geq 0$ and $\ve>0$ such that
\begin{equation}\label{eq:est3rd1}
\|S\ast \widetilde{F}(\cdot,X(\cdot))\|_{C^{\lambda}([0,T],E_p^{\eta})}\leq CT^{\ve}\|\widetilde{F}(\cdot,X(\cdot))\|_{L^q(0,T;E_p^{-\frac{1}{2}})}.
\end{equation}
We will use that by the proof of Theorem \ref{theo:SCPnsolcontB}, Assumptions \ref{assum:main}(6) holds with $\kF=\frac{1}{2}$. Hence, we obtain
\begin{align}
\mathbb{E}\|\widetilde{F}(\cdot,X(\cdot))\|^q_{L^q(0,T;E_p^{-\frac{1}{2}})}&=\mathbb{E}\int_0^T\|\widetilde{F}(s,X(s))\|^q_{E_p^{-\frac{1}{2}}}\ds\\
& \lesssim \mathbb{E}\int_0^T(1+\|X(s)\|^{q}_{E^c})\ds\\
& \lesssim 1+\mathbb{E}\|X(t)\|^{q}_{C([0,T],B))},
\end{align}
where we have used \eqref{eq:XinB} again. This implies that for each $T>0$ there exists $C'_T>0$ such that
\begin{equation}\label{eq:SasttildeF}
\left(\mathbb{E}\|S\ast \widetilde{F}(\cdot,X(\cdot))\|_{C^{\lambda}([0,T],E_p^{\eta})}^q\right)^{\frac{1}{q}}
\leq C'_T\cdot \left(1+\|X(t)\|_{L^{Kq}(\Omega; C([0,T],B))}\right)^K,
\end{equation}
and the right-hand-side is finite by \eqref{eq:XinB}.

To estimate the stochastic convolution term in \eqref{eq:proofmildsol2} we first fix $0<\alpha<\frac{1}{2}$ such that 
\[\lambda+\eta+\frac{1}{4}<\alpha-\frac{1}{q}\] 
holds. We now choose $\kappa_G\in(\frac{1}{4},\frac{1}{2})$ such that
\[\lambda+\eta+\kappa_G<\alpha-\frac{1}{q}\]
is satisfied. 

In the following we proceed similarly as done in the proof of \cite[Thm.~3.14]{KS21}. Using that Assumptions \ref{assum:main}(7) holds by the proof of Theorem \ref{theo:SCPnsolcontB}, we obtain
\begin{align}
\mathbb{E}\left\|S\diamond G(\cdot,X(\cdot))\right\|^q_{C^{\lambda}([0,T],E_p^{\eta})} &\leq T^{\ve q}\left(\int_0^T t^{-2\alpha}\dt\right)^{\frac{q}{2}}\mathbb{E}\int_0^T \left\|G(t,X(t))\right\|^q_{\gamma(E_2,E_p^{-\kappa_G})}\dt \\
&\leq T^{(\frac{1}{2}-\alpha+\ve)q}(c')^q\cdot \mathbb{E}\int_0^T\left(1+\|X(t)\|_{E^c}\right)^{\frac{k}{K}q}\dt\\
&\leq T^{(\frac{1}{2}-\alpha+\ve)q}(c')^q\cdot \mathbb{E}\int_0^T\left(1+\|X(t)\|_{B}\right)^{q}\dt\\
&\lesssim T^{(\frac{1}{2}-\alpha+\ve)q+1}(c')^q\cdot\left(1+\mathbb{E}\|X(t)\|^q_{C([0,T],B)}\right),
\end{align}
where we have used that by \eqref{eq:defkK}, $\frac{k}{K}\leq 1$ and the $E^c$- and $B$-norms coincide on $B$.
Hence, for each $T>0$ there exists constant $C''_{T}>0$ such that
\begin{equation}\label{eq:SdiamondG}
\left(\mathbb{E}\left\|S\diamond G(\cdot,X(\cdot))\right\|^q_{C^{\lambda}([0,T],E_p^{\eta})}\right)^{\frac{1}{q}}
\leq C''_{T}\cdot\left(1+\|X(t)\|_{L^{Kq}(\Omega; C([0,T],B))}\right)^{K}<+\infty.
\end{equation}
In summary, by \eqref{eq:Sxi}, \eqref{eq:SastF}, \eqref{eq:SasttildeF} and \eqref{eq:SdiamondG}, we obtain that 
\[X\in L^q(\Omega;C^{\lambda}([0,T],E_p^{\eta}))\] 
holds, hence the proof is completed.
\end{proof}

\begin{rem}
Our assumptions on the polynomials $f_{\me}$ are satisfied e.g.~for the functions coming from the classical FitzHugh-Nagumo problem (see e.g. \cite{BMZ08})
\begin{equation}\label{eq:FHN}
f_{\me}(\eta)\coloneqq \eta(\eta-1)(a_{\me}-\eta),\quad {\me}\in\mE
\end{equation}
with $a_{\me}\in (0,1)$, and also for the Allen--Cahn type nonlinearities
\begin{equation}\label{eq:AC}
f_{\me}(\eta)\coloneqq -\eta^3+\beta_{\me}^2\eta,\quad {\me}\in\mE
\end{equation}
for some (positive) constants $\beta_{\me}$ (see \cite{KS21}). Using the theory from \cite{KS21} we wouldn't have been able to treat nonlinarities of the form \eqref{eq:FHN} at all. Even to handle \eqref{eq:AC}, the problem had to be rewritten in a tricky way, see \cite[(3.33)]{KS21}.\\
The theory developed here makes possible to treat the above cases directly, and even much more general nonlinearities. We could set e.g.~a $3th$ degree polynomial on one edge and a $5th$ degree one on an other edge, satisfying Assumptions \ref{assum:fj}.
\end{rem}

We remark that in equation (\ref{eq:stochnet}a) we could have prescribed coloured noise instead of white noise on the edges of the graph, see \cite[Sec.~3.4]{KS21}. That is, we could set
\begin{equation}\label{eq:stochnetmod}
\left\{\begin{aligned}
\dot{u}_{\me}(t,x) & =  (c_{\me} u_{\me}')'(t,x)+d_{\me}(x)\cdot u_{\me}'(t,x)&\\
&-p_{\me}(x) u_{\me}(t,x)+f_{\me}(t,x,u_{\me}(t,x))&\\
&+ h_{\me}(t,x,u_{\me}(t,x))R_{\me}\frac{\partial w_{\me}}{\partial t}(t,x), &x\in (0,\ell_{\me}),\; t\in(0,T],\; \me\in\mE,
\end{aligned}
\right.
\end{equation}
with $R_{\me}\in \gamma (L^2(0,\ell_{\me}),L^p(0,\ell_{\me}))$. Then we define
\[R\coloneqq \diag (R_{\me})_{\me\in\mE}\in \gamma (E_2,E_p)\]
with 
$p\geq 2$ arbitrary. Using this, we can define the operator $G:[0,T]\times E^c\to \gamma(E_2,E_p)$ as
\[G(t,u)y\coloneqq \Gamma(t,u)Ry,\quad y\in E_2,\]
where the operator $\Gamma:[0,T]\times E^c\to \mathcal{L}(E_2)$ is defined in \eqref{eq:Gammadef}. It is easy to see that $G$ satisfies Assumptions \ref{assum:main}(7) with $\kappa_G=0$. 

If we replace (\ref{eq:stochnet}a) with \eqref{eq:stochnetmod}, Theorem \ref{theo:SCPnsolcontB} remains true as stated; that is, for $q>4$, but one may use a simpler Hilbert space machinery; that is, one may set $p=2$ in the proof. However, in the coloured noise case, Theorem \ref{theo:SCPnsolcontB} is true also for $q>2$. But this can only be shown by choosing $p>2$ large enough in the proof and hence, in this case, the Banach space arguments are crucial.

In Theorem \ref{theo:Holderreg}, if one takes $p=2$ (Hilbert space) and $q>4$, then the statement is true for  $\la+\eta>\frac14$ with
\begin{equation}\label{eq:thmHolderreglaetamod}
\la+\eta<\frac12-\frac1q
\end{equation}
instead of \eqref{eq:thmHolderreglaeta}. 
In this case $R$ will be a Hilbert-Schmidt operator whence the covariance operator of the driving process is trace-class. However, the statement of the theorem remains true for $q>2$ as well assuming \eqref{eq:thmHolderreglaetamod} instead of \eqref{eq:thmHolderreglaeta}, but only for the Banach space $E_p$ for $p$ large enough so that $\la+\eta>\frac{1}{2p}$.

\end{document}